\documentclass[12pt]{amsart}
\usepackage{amssymb, colordvi}
\usepackage{multirow}
\usepackage{graphicx}

\newtheorem{theorem}{Theorem}
\newtheorem{proposition}[theorem]{Proposition}
\newtheorem{lemma}[theorem]{Lemma}

\newtheorem{ex}[theorem]{Example}
\newtheorem{rem}[theorem]{Remark}

\newcounter{FNC}[page]
\def\fauxfootnote#1{{\addtocounter{FNC}{2}$^\fnsymbol{FNC}$%
     \let\thefootnote\relax\footnotetext{$^\fnsymbol{FNC}$#1}}}

\renewcommand{\P}{\mathbb{P}}
\newcommand{\C}{\mathbb{C}}
\newcommand{\R}{\mathbb{R}}
\newcommand{\Z}{\mathbb{Z}}

\newcommand{\calA}{\mathcal{A}}

\newcommand{\calW}{\mathcal{W}}
\title{Bounds on the number of real solutions to polynomial equations} 

\author{Daniel J.~Bates}
\address{Institute for Mathematics and its Applications\\
         University of Minnesota\\
         114 Lind Hall\\
         207 Church Street S.E.\\
         Minneapolis, MN 55455-0436\\
         USA}

\email{dbates1@nd.edu}
\urladdr{http://www.nd.edu/~dbates1/}

\author{Fr\'ed\'eric Bihan}
\address{Laboratoire de Math\'ematiques\\
         Universit\'e de Savoie\\
         73376 Le Bourget-du-Lac Cedex\\
         France}

\email{Frederic.Bihan@univ-savoie.fr}
\urladdr{http://www.lama.univ-savoie.fr/\~{}bihan/}

\author{Frank Sottile}
\address{Department of Mathematics\\
         Texas A\&M University\\
         College Station\\
         Texas \ 77843\\
         USA}
\email{sottile@math.tamu.edu}
\urladdr{http://www.math.tamu.edu/\~{}sottile/}

\thanks{Bates and Sottile supported by the Institute for Mathematics and its Applications.}
\thanks{Sottile supported by the NSF CAREER grant DMS-0538734.}  



\begin{document}

\begin{abstract}
 We use Gale duality for complete intersections and adapt the 
 proof of the fewnomial bound for positive solutions to obtain the bound
\[
   \frac{e^{\Red{4}}+3}{4}2^{\binom{k}{2}}n^k
\]
 for the number of non-zero real solutions to a system of $n$ polynomials 
 in $n$ variables having $n{+}k{+}1$ monomials whose exponent vectors generate
 a subgroup of $\Z^n$ of odd index.
 This bound only exceeds the bound for positive solutions
 by the constant factor $(e^4+3)/(e^2+3)$ and it is 
 asymptotically sharp for $k$ fixed and $n$ large.
\end{abstract}
\maketitle

%
\section*{Introduction}
%
In~\cite{BBS}, the sharp bound of $2n{+}1$ was obtained for the number of 
non-zero real solutions to a system of $n$ polynomial equations in $n$ 
variables having $n{+}2$ monomials whose exponents affinely span the lattice $\Z^n$.
In~\cite{Bi07}, the sharp bound of $n{+}1$ was given for the positive solutions
to such a system of equations.
This last bound was generalized in~\cite{BS07}, which showed that the 
number of positive solutions to a system of $n$ polynomial equations in $n$
variables having $n{+}k{+}1$ monomials was less than 
\[
   \frac{e^2+3}{4} 2^{\binom{k}{2}}n^k\,,
\]
which is asymptotically sharp for $k$ fixed and $n$ large~\cite{BRS07a}.
This dramatically improved Khovanskii's fewnomial bound~\cite{Kh80} 
of $2^{\binom{n+k}{2}}(n+1)^{n+k}$.

We give a bound for all non-zero real solutions.
Under the assumption that the exponent vectors $\calW$ span a subgroup of $\Z^n$ of odd  
index, we show that the number of non-degenerate non-zero real solutions
to a system of polynomials with support $\calW$ is less than 
 \begin{equation}\label{Eq:New_bound}
   \frac{e^{\Red{4}}+3}{4} 2^{\binom{k}{2}}n^k\ .
 \end{equation}
The novelty is that this bound exceeds the bound for solutions in the positive
orthant by a fixed constant factor $(e^4+3)/(e^2+3)$, rather than by a factor of $2^n$,
which is the number of orthants.
By the construction in~\cite{BRS07a}, it is asymptotically sharp for $k$ fixed and $n$
large. 

We follow the outline of~\cite{BS07}---we use Gale duality for real
complete intersections~\cite{BS_Gale} and then bound the number of solutions to the dual
system of master functions.
The key idea is that including solutions in all chambers in a complement of
an arrangement of hyperplanes in $\R\P^k$, rather than in just one chamber as
in~\cite{BS07}, does not increase our estimate on the number of solutions very much. 
This was discovered while implementing a numerical continuation algorithm for computing
the positive solutions to a system of polynomials~\cite{BaSo}. 
That algorithm was improved by this discovery to one which finds all real solutions.
It does  so without computing complex solutions and is based on~\cite{BS07} and the
results of this paper.
Its complexity depends on~\eqref{Eq:New_bound}, and not on the number of
complex solutions. 

We state our main theorem in Section 1 and then use Gale duality to reduce it to a
statement about systems of master functions, which we prove in Section 2.

%
%
%
\section{Gale duality for systems of sparse polynomials}

Let $\calW=\{w_0=0,w_1,\dotsc,w_{n{+}k}\}\subset\Z^n$  be a collection of $n{+}k{+}1$
integer vectors ($|\calW|=n{+}k{+}1$), which correspond to monomials in variables
$x_1,\dotsc,x_n$. 
A (Laurent) polynomial $f$ with \Blue{{\it support $\calW$}} is a real linear
combination of monomials with exponents from $\calW$,
 \begin{equation}\label{Eq:sparse_poly}
     f(x_1,\dotsc,x_n)\ =\ \sum_{i=0}^{n+k} c_i x^{w_i}
      \qquad\mbox{with}\ c_i\in\R\,.
 \end{equation}
A \Blue{{\it system with support $\calW$}} is a system of polynomial equations
 \begin{equation}\label{Eq:sparse_system}
   f_1(x_1,\dotsc,x_n)\ =\  f_2(x_1,\dotsc,x_n)
    \ =\ \dotsb \ =\ f_n(x_1,\dotsc,x_n)\ =\ 0\,,
 \end{equation}
where each polynomial $f_i$ has support $\calW$.
Since multiplying every polynomial in~\eqref{Eq:sparse_system} by a monomial 
$x^\alpha$ does not change the set of non-zero solutions but translates $\calW$ by the 
vector $\alpha$, we see that it was no loss of generality to assume that $0\in\calW$.

The system~\eqref{Eq:sparse_system} has infinitely many solutions if $\calW$ does not span
$\R^n$.
We say that $\calW$ \Blue{{\it spans $\Z^n$ mod $2$}} if the $\Z$-linear span of $\calW$
is a subgroup of $\Z^n$ of odd index.

\begin{theorem}\label{Thm:one}
 Suppose that $\calW$ spans $\Z^n$ mod $2$ and $|\calW|=n{+}k{+}1$.
 Then there are fewer than~$\eqref{Eq:New_bound}$ non-degenerate non-zero real solutions 
 to a sparse system~$\eqref{Eq:sparse_system}$ with support $\calW$.
\end{theorem}

The importance of this bound for the number of real solutions is that it has a completely
different character than Kouchnirenko's bound for the number of complex solutions.

\begin{proposition}[Kouchnirenko~\cite{BKK}]
 The number of non-degenerate solutions in $(\C^\times)^n$ to a
 system~$\eqref{Eq:sparse_system}$ with support $\calW$ is no more than
 $n!\mbox{\rm vol}(\mbox{\rm conv}(\calW))$.
\end{proposition}

Here, $\mbox{\rm vol}(\mbox{\rm conv}(\calW))$ is the Euclidean volume of the convex
hull of $\calW$. 

Perturbing coefficients of the polynomials in~\eqref{Eq:sparse_system} so that
they define a complete intersection in $(\C^\times)^n$
can only increase the number of non-degenerate solutions.
Thus it suffices to prove Theorem~\ref{Thm:one} under this assumption.
Such a complete intersection is equivalent to a complete intersection of 
master functions in a hyperplane complement~\cite{BS_Gale}.

Let $\R^{n+k}$ have coordinates $z_1,\dotsc,z_{n+k}$. 
A polynomial~\eqref{Eq:sparse_poly} with support $\calW$ is the pullback
$\Phi_\calW^*(\Lambda)$ of the degree 1 polynomial
$\Lambda:= c_0+c_1z_1+\dotsb+c_{n+k}z_{n+k}$ 
along the map 
\[
   \Phi_\calW\ \colon\ (\R^\times)^n\ni x\ \longmapsto\ 
    (x^{w_i}\mid i=1,\dotsc,n{+}k)\in \R^{n+k}\,.
\]
If we let $\Lambda_1,\dotsc,\Lambda_n$ be the degree 1 polynomials which pull back 
to the polynomials in the system~\eqref{Eq:sparse_system}, then they cut out an affine
subspace $L$ of $\R^{n+k}$ of dimension $k$.

Let $\{p_i\mid i=1,\dotsc,n{+}k\}$ be degree 1 polynomials
on $\R^k$ which induce an isomorphism  between $\R^k$ and $L$, 
 \[
   \Psi_p\ \colon\  \R^k\ni y\ \longmapsto\ (p_1(y),\dotsc,p_{n+k}(y))\in
   L\subset\R^{n+k}\,.
 \]
Let $\calA\subset\R^k$ be the arrangement of hyperplanes defined by the vanishing of the 
$p_i(y)$.
This is the pullback along $\Psi_p$ of the coordinate hyperplanes of $\R^{n+k}$.

The image $\Phi_\calW((\R^\times)^n)$ inside of the torus $(\R^\times)^{n{+}k}$
has equations
\[
   z^{\beta_1}\ =\ z^{\beta_2}\ =\ \dotsb\ =\ z^{\beta_k}\ =\ 1\,,
\]
where the \Blue{{\it weights}} $\{\beta_1,\dotsc,\beta_k\}$ form a basis for the
$\Z$-submodule of $\Z^{n+k}$ of linear relations among the vectors $\calW$.
To these data, we associate a system of 
\Blue{{\it master functions}} on the complement $M_\calA$ of
the arrangement $\calA$ of $\R^k$,
 \begin{equation}\label{Eq:Master}
  p(y)^{\beta_1}\ =\ p(y)^{\beta_2}\ =\ \dotsb\ =\ p(y)^{\beta_k}\ =\ 1\,.
 \end{equation}
Here, if $\beta=(b_1,\dotsc,b_{n+k})$ then 
$p^\beta:=p_1(y)^{b_1}\dotsb p_{n{+}k}(y)^{b_{n{+}k}}$.

A basic result of~\cite{BS_Gale} is that if $\calW$ spans $\Z^n$ modulo 2 and
either of the systems~\eqref{Eq:sparse_system} or~\eqref{Eq:Master} defines a complete
intersection, then the other defines a complete intersection and the maps $\Phi_\calW$ and
$\Psi_p$ induce isomorphisms between the two solution sets, as 
analytic subschemes of $(\R^\times)^n$ and $M_\calA$.
Since we assumed that the system~\eqref{Eq:sparse_system} is general, these
hypotheses hold and the arrangement is \Blue{{\it essential}} in that the 
polynomials $p_i$ span the space of all degree 1 polynomials on $\R^k$.

\begin{theorem}
 A system~$\eqref{Eq:Master}$ of master functions in the complement 
 of an essential arrangement of $n{+}k$ hyperplanes in $\R^k$ 
 has at most~$\eqref{Eq:New_bound}$ non-degenerate real solutions.
\end{theorem}

We actually prove a bound for a more general system than~\eqref{Eq:Master},
namely for
\[
  p(z)^{2\beta_1}\ =\ p(z)^{2\beta_2}\ =\ \dotsb\ =\ p(z)^{2\beta_k}\ =\ 1\,.
\]
We write this more general system as 
 \begin{equation}\label{Eq:gen_Master}
  |p(z)|^{\beta_1}\ =\ |p(z)|^{\beta_2}\ =\ \dotsb\ =\ |p(z)|^{\beta_k}\ =\ 1\,.
 \end{equation}
In a system of this form we may have real number weights
$\beta_i\in\R^{n+k}$. 
We give the strongest form of our theorem.

\begin{theorem}\label{Thm:ultimate}
  A system of the form~$\eqref{Eq:gen_Master}$ with real weights $\beta_i$ in 
  the complement of an essential arrangement of $n{+}k$ hyperplanes in $\R^k$ 
 has at most~$\eqref{Eq:New_bound}$ non-degenerate real solutions.
\end{theorem}

%
%
%
\section{Proof of Theorem~\ref{Thm:ultimate}}

We follow~\cite{BS07} with minor, but
important, modifications.
Perturbing the polynomials $p_i(y)$ and the weights $\beta_j$ will not decrease the 
number of non-degenerate real solutions in $M_\calA$.
This enables us to make the following assumptions.

The arrangement $\calA^+\subset\R\P^k$, where we add the hyperplane at infinity, is
general in that every $j$ hyperplanes of $\calA^+$ meet in a $(k{-}j)$ dimensional linear
subspace, called a \Blue{{\it codimension $j$ face of $\calA$}}.
If $B$ is the matrix whose columns are the weights $\beta_1,\dotsc,\beta_k$, then the
entries of $B$ are rational numbers and no minor of $B$ vanishes.
This last technical condition as well as the freedom to further perturb the $\beta_j$
and the $p_i$ are necessary for the results in~\cite[Section 3]{BS07} upon which we rely.

For functions $f_1,\dotsc,f_j$ on $M_\calA$, let 
$V(f_1,\dotsc,f_j)$ be the subvariety they define.
Suppose that $\beta_j=(b_{1,j},\dotsc,b_{n+k,j})$.
For each $j=1,\dotsc,k$, define
\[
   \psi_j(y)\ :=\ \sum_{i=1}^{n+k} b_{i,j} \log |p_i(y)|\,.
\]
Then~\eqref{Eq:gen_Master} is equivalent to 
$\psi_1(y)=\dotsb=\psi_k(y)=0$. 
Inductively define $\Gamma_k,\Gamma_{k-1},\dotsc,\Gamma_1$ by
\[
   \Gamma_j\ :=\ \mbox{\rm Jac}(\psi_1,\dotsc,\psi_j,\ 
                  \Gamma_{j{+}1},\dotsc,\Gamma_k)\,,
\]
the Jacobian determinant of $\psi_1,\dotsc,\psi_j,\Gamma_{j{+}1},\dotsc,\Gamma_k$.
Set
 \[
     C_j\ :=\  V(\psi_1,\dotsc,\psi_{j-1},\ \Gamma_{j{+}1},\dotsc,\Gamma_k)\,,
 \]
which is a curve in  $M_\calA$.

Let $\flat(C)$ be the number of unbounded components of a curve $C\subset M_\calA$.
We have the estimate from~\cite{BS07}, which is
a consequence of the Khovanskii-Rolle Theorem,
 \begin{equation}\label{Eq:Kh-Ro_estimate}
   |V(\psi_1,\dotsc,\psi_k)|\ \leq\ 
    \flat(C_k)+\dotsb+\flat(C_1)\ +\ 
   |V(\Gamma_1,\dotsc,\Gamma_k)|\,.
 \end{equation}
Here, $|S|$ is the cardinalty of the set $S$.
We estimate these quantities.

\begin{lemma}\label{L:estimate}
  \mbox{\ }
 \begin{enumerate}
  \item[$(1)$]
       $|V(\Gamma_1,\dotsc,\Gamma_k)|\leq  2^{\binom{k}{2}}n^k$.

  \item[$(2)$]
    $C_j$ is a smooth curve and 
\[
    \flat(C_j)\ \ \leq\ \ 
      \frac{1}{2}2^{\binom{k-j}{2}}n^{k-j}\tbinom{n{+}k{+}1}{j}\cdot 2^j
      \ \ \leq\ \ \frac{1}{2}2^{\binom{k}{2}}n^k\cdot\frac{2^{2j-1}}{j!}\,.
\]
 \end{enumerate}
\end{lemma}

\noindent{\it Proof of Theorem~$\ref{Thm:ultimate}$.}
  By~\eqref{Eq:Kh-Ro_estimate} and  Lemma~\ref{L:estimate},
  we have
\[
  |V(\psi_1,\dotsc,\psi_k)|\ \ \leq\ \ 
  2^{\binom{k}{2}}n^k \Bigl( 1+ \frac{1}{4}\sum_{j=1}^k \frac{4^j}{j!}\Bigr)
  \ \ <\  \  2^{\binom{k}{2}}n^k \cdot \frac{e^4+3}{4}\,.
   \eqno{\Box}\bigskip
\]


\begin{proof}[Proof of Lemma~$\ref{L:estimate}$]
 The bound (1) is from Lemma~3.4 of~\cite{BS07}.
 Statements analogous to (2) for $\widetilde{C}_j$, the restriction of $C_j$
 to a single chamber (connected component) of $M_\calA$, were established in 
 Lemma~3.4 and the proof of Lemma~3.5 in~\cite{BS07}:
 \begin{equation}\label{Eq:tildeC}
    \flat(\widetilde{C}_j)\ \ \leq\ \  
       \frac{1}{2}2^{\binom{k-j}{2}}n^{k-j}\tbinom{n{+}k{+}1}{j}
        \ \leq\ \ \frac{1}{2}2^{\binom{k}{2}}n^k\cdot\frac{2^{j-1}}{j!}\ .
 \end{equation}
 The bound we claim for $\flat(C_j)$ has an extra factor of $2^j$.
 {\it A priori} we would expect to multiply this bound~\eqref{Eq:tildeC} by the number
 of chambers of $M_\calA$ to obtain a bound for $\flat(C_j)$, but
 the correct factor is only $2^j$.

 We work in $\R\P^k$ and use the extended hyperplane arrangement 
 $\calA^+$, as we will need points in the closure of $C_j$ in  $\R\P^k$.
 The first inequality in~\eqref{Eq:tildeC} for $\flat(\widetilde{C}_j)$ arises as each
 unbounded component of $\widetilde{C}_j$ meets $\calA^+$ in two distinct points
 (this accounts for the factor $\frac{1}{2}$) which are points
 of codimension $j$ faces where the polynomials
\[
  F_i(y)\ :=\ \Gamma_{k-i}(y)\cdot \Bigl(\prod_{i=1}^{n{+}k} p_i(y)\Bigr)^{2^i}
\]
 for $i=0,\dotsc,k-j-1$ vanish.
 (By Lemma 3.4(1) of~\cite{BS07}, $F_i$ is a polynomial of degree $2^in$.)
 The genericity of the weights and the linear polynomials $p_i(y)$ 
 imply that these points will lie on faces of codimension $j$ but not
 of higher codimension.
 The factor $2^{\binom{k-j}{2}}n^{k-j}$ is the B\'ezout number of the 
 system $F_0=\dotsb=F_{k-j-1}$ on a given codimension $j$ plane, and 
 there are exactly $\binom{n{+}k{+}1}{j}$ codimension $j$ faces of $\calA^+$.

 At each of these points, $C_j$ will have one branch in each chamber of 
 $M_\calA$ incident on that point.
 Since the hyperplane arrangement $\calA^+$ is general there will be exactly 
 $2^j$ such chambers.
\end{proof}

\providecommand{\bysame}{\leavevmode\hbox to3em{\hrulefill}\thinspace}
\providecommand{\MR}{\relax\ifhmode\unskip\space\fi MR }
\providecommand{\MRhref}[2]{%
  \href{http://www.ams.org/mathscinet-getitem?mr=#1}{#2}
}
\providecommand{\href}[2]{#2}

\bibliographystyle{amsplain}
\bibliography{bibl}
 
\end{document}